\documentclass[12pt]{article}
\usepackage{amsmath,amssymb,amsthm}
\numberwithin{equation}{section}
\usepackage{units}
\usepackage{color}
\usepackage[T1]{fontenc}
\usepackage[utf8]{inputenc}
\usepackage{authblk}
\usepackage{bm}
\usepackage{graphicx}

\usepackage{datetime}

\usepackage[colorlinks=true,
linkcolor=webgreen,
filecolor=webbrown,
citecolor=webgreen]{hyperref}

\definecolor{webgreen}{rgb}{0,.5,0}
\definecolor{webbrown}{rgb}{.6,0,0}

\usepackage[toc,page]{appendix}

\newtheorem{thm}{Theorem}
\newtheorem{theorem}[thm]{Theorem}
\newtheorem{lemma}{Lemma}

\newtheorem{corollary}[thm]{Corollary}

\title{Partial sums and generating functions for powers of second order sequences with indices in arithmetic progression}
\author[]{Kunle Adegoke \\\href{mailto:kunle.adegoke@yandex.com}{\tt kunle.adegoke@yandex.com}}

\affil{Department of Physics and Engineering Physics, \mbox{Obafemi Awolowo University}, 220005 Ile-Ife, Nigeria}

\begin{document}
\date{}

\maketitle

\begin{abstract} \noindent The sums $\sum_{j = 0}^k {u_{rj + s}^{2n}z^j }$, $\sum_{j = 0}^k {u_{rj + s}^{2n-1}z^j }$, $\sum_{j = 0}^k {v_{rj + s}^{n}z^j }$ and $\sum_{j = 0}^k {w_{rj + s}^{n}z^j }$ are evaluated; where $n$ is any positive integer, $r$, $s$ and $k$ are any arbitrary integers, $z$ is arbitrary, $(u_i)$ and $(v_i)$ are the Lucas sequences of the first kind, and of the second kind, respectively; and $(w_i)$ is the Horadam sequence. Pantelimon St\u anic\u a set out to evaluate the sum $\sum_{j = 0}^k {w_j^n z^j }$. His solution is not complete because he made the assumption that $w_0=0$, thereby giving effectively only the partial sum for $(u_i)$, the Lucas sequence of the first kind. 
\end{abstract}
\section{Introduction}
The Horadam sequence \cite{horadam65} $(w_n) = \left(w_n(a,b; p, q)\right)$ is defined, for all integers, by the recurrence relation
\begin{equation}\label{eq.vhrb5b3}
w_0  = a,\,w_1  = b;\,w_n  = pw_{n - 1}  - qw_{n - 2}\, (n \ge 2)\,;
\end{equation}
with
\begin{equation}
w_{ - n}  = \left( {\frac{{(ap - b)u_n  - aqu_{n - 1} }}{{bu_n  - aqu_{n - 1} }}} \right)q^{ - n} w_n\,,
\end{equation}
or, equivalently,
\begin{equation}\label{eq.xzxwnfu}
w_{ - n}  = q^{-n}(av_n  - w_n)\,;
\end{equation}
where $a$, $b$, $p$ and $q$ are arbitrary complex numbers, with $p\ne 0$ and $q\ne 0$; and $(u_n(p,q))=(w_n(0,1;p,q))$ and $(v_n(p,q))=(w_n(2,p;p,q))$ are Lucas sequences of the first kind, and of the second kind, respectively. The most well-known Lucas sequences are the Fibonacci sequence, $(f_n)=(u_n(1,-1))$ and the sequence of Lucas numbers, $(l_n)=(v_n(1,-1))$.

\medskip

Denote by $\alpha$ and $\beta$, with $\alpha\ne\beta$, the zeros of the characteristic polynomial $x^2-px+q$ of the Horadam and Lucas sequences. Then
\begin{equation}\label{eq.ufpxoag}
\alpha=\frac{p+\sqrt{p^2-4q}}2,\quad \beta=\frac{p-\sqrt{p^2-4q}}2\,,
\end{equation}
\begin{equation}\label{eq.dv3pphj}
\alpha+\beta=p,\quad\alpha-\beta=\sqrt{p^2-4q}\quad\mbox{and }\alpha\beta=q\,.
\end{equation}
The Binet-like formulas for $u_n$, $v_n$ and $w_n$ are
\begin{equation}\label{eq.djr1ak1}
u_n=\frac{\alpha^n-\beta^n}{\alpha-\beta},\quad v_n=\alpha^n+\beta^n\,,
\end{equation}
and
\begin{equation}\label{eq.ffgsygd}
w_n = A\alpha ^n  + B\beta ^n\,,
\end{equation}
where
\begin{equation}
A=\frac{{b - a\beta}}{{\alpha  - \beta }},\quad B=\frac{{a\alpha  - b}}{{\alpha  - \beta }}\,.
\end{equation}
Properties of Lucas sequences can be found in the book by Ribenboim \cite[Chapter 1]{ribenboim}. The Mathworld \cite{mathworld_lucas} and Wikipedia \cite{wiki_lucas} articles are also good sources of information on the subject, with many references to useful materials. The books by Koshy \cite{koshy} and by Vajda \cite{vajda} are excellent reference materials on Fibonacci numbers and Lucas numbers.

\medskip

St\u anic\u a \cite{stanica03} set out to evaluate the sum $\sum_{j = 0}^k {w_j^n z^j }$. His solution is not complete because he made the assumption (in our notation) that $w_0=a=0$, thereby giving effectively only the partial sum for $(u_i)$, the Lucas sequence of the first kind. 

\medskip

Our main goal in this paper is to establish the following result, for $n$ a positive integer and $r$, $s$ and $k$ arbitrary integers:
\[
\begin{split}
&2\sum\limits_{j = 0}^k {w_{rj+s}^n z^j } \\
&\quad=z^{k + 2}\sum\limits_{i = 0}^n {(AB)^iq^{si}\binom ni\frac{{A^{n - 2i} \alpha^{s(n - 2i) + rn(k + 1) - rik} \beta^{r(ik + n)}  + B^{n - 2i} \alpha^{r(ik + n)} \beta^{s(n - 2i) + rn(k + 1) - rik} }}{{q^{rn} z^2  - q^{ri} v_{r(n - 2i)}z + 1}} } \\
&\qquad-z^{k + 1}\sum\limits_{i = 0}^n {(AB)^iq^{si}\binom ni\frac{{A^{n - 2i} \alpha^{s(n - 2i) + (rn - ri)(k + 1)} \beta^{ri(k + 1)}  + B^{n - 2i} \alpha^{ri(k + 1)} \beta^{s(n - 2i) + (rn - ri)(k + 1)} }}{{q^{rn} z^2  - q^{ri} v_{r(n - 2i)}z + 1}} } \\
&\qquad-z\sum\limits_{i = 0}^n {(AB)^iq^{si}\binom ni\frac{{A^{n - 2i} \alpha^{s(n - 2i) + ri} \beta^{r(n - i)}  + B^{n - 2i} \alpha^{r(n - i)} \beta^{s(n - 2i) + ri} }}{{q^{rn} z^2  - q^{ri} v_{r(n - 2i)}z + 1}}} \\
&\quad\qquad+\sum\limits_{i = 0}^n {(AB)^iq^{si}\binom ni\frac{{A^{n - 2i} \alpha^{s(n - 2i)}  + B^{n - 2i} \beta^{s(n - 2i)} }}{{q^{rn} z^2  - q^{ri} v_{r(n - 2i)}z + 1}}}\,.
\end{split}
\]
In particular, we will derive
\[
\sum_{j = 0}^k {w_{rj + s} z^j}  = \frac{{q^r w_{rk + s} z^{k + 2}  - w_{rk + r + s} z^{k + 1}  - q^r w_{s-r}z + w_s }}{{q^r z^2  - v_r z + 1}}\,,
\]
yielding immediately,
\[
\sum_{j = 0}^\infty {w_{rj + s} z^j}  = \frac{ w_s - q^r w_{s-r}z}{q^r z^2  - v_r z + 1}
\]
as the generating function of Horadam numbers with indices in arithmetic progression.

\medskip

We require the following two algebraic identities:
\begin{lemma}\label{lemma.sliph2a}
The following identity holds for arbitrary $f$, $g$, $x$, $y$ and $z$ and integers $k$, $r$ and~$s$:
\[
\begin{split}
&\sum_{j = 0}^k {(fx^{rj + s}  + g y^{rj + s} )z^j }\\
&\quad= \frac{(xy)^r (fx^{rk + s}  + g y^{rk + s} )}{ {(xy)^r z^2  - (x^r  + y^r )z + 1} }z^{k + 2}  - \frac{fx^{rk + r + s}  + g y^{rk + r + s} }{ {(xy)^r z^2  - (x^r  + y^r )z + 1} }z^{k + 1}\\ 
&\qquad\quad- \frac{(xy)^r (fx^{s - r}  + g y^{s - r} )}{ {(xy)^r z^2  - (x^r  + y^r )z + 1} }z + \frac{fx^s  + g y^s}{ {(xy)^r z^2  - (x^r  + y^r )z + 1} }\,.
\end{split}
\]
\end{lemma}\label{lemma.hb5fyhe}
\begin{proof}
The identity expresses the linear combination of the following geometric progression summation identities:
\begin{equation}\label{eq.ahuskzm}
\sum_{j = 0}^k {x^{rj + s} z^j }  = \frac{{x^{rk + r + s} z^{k + 1}  - x^s }}{{x^r z - 1}}\,,
\end{equation}
\begin{equation}\label{eq.o3wjsin}
\sum_{j = 0}^k {y^{rj + s} z^j }  = \frac{{y^{rk + r + s} z^{k + 1}  - y^s }}{{y^r z - 1}}\,.
\end{equation}
\end{proof}
Dropping terms proportional to $z^k$, in the limit as $k$ approaches infinity, we have
\begin{equation}\label{eq.rrfnlqs}
\sum_{j = 0}^\infty {(fx^{rj + s}  + g y^{rj + s} )z^j }=\frac{fx^s  + g y^s - (xy)^r (fx^{s - r}  + g y^{s - r} )z}{{(xy)^r z^2  - (x^r  + y^r )z + 1}}\,.
\end{equation}
\begin{lemma}\label{lemma.l8djsvx}
The following identity holds for arbitrary $f$, $g$, $x$, $y$ and $z$ and integers $k$, $r$ and~$s$ and nonnegative integer $n$:
\[
\begin{split}
&2\sum\limits_{j = 0}^k {(fx^{rj + s}  + gy^{rj + s} )^n z^j } \\
&\quad=z^{k + 2}\sum\limits_{i = 0}^n {(fg)^i(xy)^{si}\binom ni\frac{{f^{n - 2i} x^{s(n - 2i) + rn(k + 1) - rik} y^{r(ik + n)}  + g^{n - 2i} x^{r(ik + n)} y^{s(n - 2i) + rn(k + 1) - rik} }}{{(xy)^{rn} z^2  - (xy)^{ri} (x^{r(n - 2i)}  + y^{r(n - 2i)} )z + 1}} } \\
&\qquad-z^{k + 1}\sum\limits_{i = 0}^n {(fg)^i(xy)^{si}\binom ni\frac{{f^{n - 2i} x^{s(n - 2i) + (rn - ri)(k + 1)} y^{ri(k + 1)}  + g^{n - 2i} x^{ri(k + 1)} y^{s(n - 2i) + (rn - ri)(k + 1)} }}{{(xy)^{rn} z^2  - (xy)^{ri} (x^{r(n - 2i)}  + y^{r(n - 2i)} )z + 1}} } \\
&\qquad-z\sum\limits_{i = 0}^n {(fg)^i(xy)^{si}\binom ni\frac{{f^{n - 2i} x^{s(n - 2i) + ri} y^{r(n - i)}  + g^{n - 2i} x^{r(n - i)} y^{s(n - 2i) + ri} }}{{(xy)^{rn} z^2  - (xy)^{ri} (x^{r(n - 2i)}  + y^{r(n - 2i)} )z + 1}}} \\
&\quad\qquad+\sum\limits_{i = 0}^n {(fg)^i(xy)^{si}\binom ni\frac{{f^{n - 2i} x^{s(n - 2i)}  + g^{n - 2i} y^{s(n - 2i)} }}{{(xy)^{rn} z^2  - (xy)^{ri} (x^{r(n - 2i)}  + y^{r(n - 2i)} )z + 1}}}\,.
\end{split}
\]
\end{lemma}
\begin{proof}
By the binomial theorem and a change of the order of summation, we have
\begin{equation}\label{eq.bycqcom}
\sum\limits_{j = 0}^k {(fx^{rj + s}  + gy^{rj + s} )^{n} z^j }  = \sum\limits_{i = 0}^{n} {\binom{n}i\sum\limits_{j = 0}^k {\left(fx^{rj + s}\right)^i \left(gy^{rj + s}\right)^{(n - i)} z^j } }\,,
\end{equation}
which, by interchanging $i$ and $n-i$ in the summand on the right hand side, can also be written
\begin{equation}\label{eq.ix4pc01}
\sum\limits_{j = 0}^k {(fx^{rj + s}  + gy^{rj + s} )^{n} z^j }  = \sum\limits_{i = 0}^{n} {\binom{n}i\sum\limits_{j = 0}^k {\left(fx^{rj + s}\right)^{(n - i)} \left(gy^{rj + s}\right)^i z^j } }\,.
\end{equation}
Adding \eqref{eq.bycqcom} and \eqref{eq.ix4pc01}, we have
\[
\begin{split}
&2\sum\limits_{j = 0}^k {(fx^{rj + s}  + gy^{rj + s} )^n z^j }\\
&\quad = \sum\limits_{i = 0}^n {\binom ni\sum\limits_{j = 0}^k {(fgx^{rj + s} y^{rj + s} )^i \left( {(fx^{rj + s} )^{n - 2i}  + (gy^{rj + s} )^{n - 2i} } \right)z^j } }\\ 
&\quad= \sum\limits_{i = 0}^n {(fg)^i (xy)^{si} \binom ni\left( {(fx^s )^{n - 2i} \sum\limits_{j = 0}^k {(x^{r(n - i)} y^{ri} z)^j }  + (gy^s )^{n - 2i} \sum\limits_{j = 0}^k {(x^{ri} y^{r(n - i)} z)^j } } \right)}\\ 
&\quad= \sum\limits_{i = 0}^n {(fg)^i (xy)^{si} \binom ni\left( {(fx^s )^{n - 2i} \frac{{(x^{r(n - i)} y^{ri} z)^{k + 1}  - 1}}{{x^{r(n - i)} y^{ri} z - 1}} + (gy^s )^{n - 2i} \frac{{(x^{ri} y^{r(n - i)} z)^{k + 1}  - 1}}{{x^{ri} y^{r(n - i)} z - 1}}} \right)}\,, 
\end{split}
\]
from which the identity of the theorem follows.
\end{proof}
Dropping terms proportional to $z^k$, in the limit as $k$ approaches infinity, we have
\begin{equation}\label{eq.yjgu05n}
\begin{split}
&2\sum\limits_{j = 0}^\infty {(fx^{rj + s}  + gy^{rj + s} )^n z^j } \\
&\qquad-z\sum\limits_{i = 0}^n {(fg)^i(xy)^{si}\binom ni\frac{{f^{n - 2i} x^{s(n - 2i) + ri} y^{r(n - i)}  + g^{n - 2i} x^{r(n - i)} y^{s(n - 2i) + ri} }}{{(xy)^{rn} z^2  - (xy)^{ri} (x^{r(n - 2i)}  + y^{r(n - 2i)} )z + 1}}} \\
&\quad\qquad+\sum\limits_{i = 0}^n {(fg)^i(xy)^{si}\binom ni\frac{{f^{n - 2i} x^{s(n - 2i)}  + g^{n - 2i} y^{s(n - 2i)} }}{{(xy)^{rn} z^2  - (xy)^{ri} (x^{r(n - 2i)}  + y^{r(n - 2i)} )z + 1}}}\,.
\end{split}
\end{equation}
\section{Main results}
\begin{theorem}[\textbf{Sum of Horadam numbers with indices in arithmetic progression}]\label{theorem.ru08riv}
The following identities hold for integers $r$, $k$ and $s$ and arbitrary $z$:
\begin{equation}\label{eq.ysk3t24}
\sum_{j = 0}^k {u_{rj + s}z^j }  = \frac{{q^r u_{rk + s} z^{k + 2}  - u_{rk + r + s} z^{k + 1}  + q^s u_{r - s}z + u_s }}{{q^r z^2  - v_r z + 1}}\,,
\end{equation}
\begin{equation}\label{eq.o00kpj6}
\sum_{j = 0}^k {v_{rj + s} z^j }  = \frac{{q^r v_{rk + s} z^{k + 2}  - v_{rk + r + s} z^{k + 1}  - q^s v_{r-s}z + v_s }}{{q^r z^2  - v_r z + 1}}\,,
\end{equation}
\begin{equation}\label{eq.e9jfrv6}
\sum_{j = 0}^k {w_{rj + s} z^j }  = \frac{{q^r w_{rk + s} z^{k + 2}  - w_{rk + r + s} z^{k + 1}  - q^rw_{s-r}z + w_s }}{{q^r z^2  - v_r z + 1}}\,.
\end{equation}
\end{theorem}
\begin{proof}
Choose $(x,y)=(\alpha,\beta)$ in Lemma \ref{lemma.sliph2a}, with $(f,g)=(1,-1)$ to get identity \eqref{eq.ysk3t24}, with $(f,g)=(1,1)$ to get identity \eqref{eq.o00kpj6} and with $(f,g)=(A,B)$ to get identity \eqref{eq.e9jfrv6}.
\end{proof}
Using identity \eqref{eq.rrfnlqs}, we have the results stated in Corollary \ref{cor.lxtsw7r}.
\begin{corollary}[\textbf{Generating functions for Lucas sequences and for the Horadam sequence with indices in arithmetic progression}]\label{cor.lxtsw7r}
The following identities hold for integers $r$ and $s$:
\begin{equation}
\sum_{j = 0}^\infty {u_{rj + s}z^j }  = \frac{q^s u_{r - s}z + u_s }{{q^r z^2  - v_r z + 1}}\,,
\end{equation}
\begin{equation}
\sum_{j = 0}^\infty {v_{rj + s} z^j }  = \frac{ - q^s v_{r-s}z + v_s }{{q^r z^2  - v_r z + 1}}\,,
\end{equation}
\begin{equation}
\sum_{j = 0}^\infty {w_{rj + s} z^j }  = \frac{-q^r w_{s-r}z + w_s }{{q^r z^2  - v_r z + 1}}\,.
\end{equation}
\end{corollary}
\begin{theorem}[\textbf{Sums of powers of the terms of Lucas sequences with indices in arithmetic progression}]
The following identities hold for integers $r$, $s$ and $k$, nonnegative integer $n$ and arbitrary $z$:
\begin{equation}\label{eq.fxcbpcf}
\begin{split}
2(p^2 -4q)^n\sum\limits_{j = 0}^k {u_{rj + s}^{2n} z^j }  =& \sum\limits_{i = 0}^{2n} {(-1)^i\binom {2n}iq^{si} \frac{{q^{r(2n + ki)} v_{(rk + s)(2n - 2i)} z^{k + 2}  - q^{ri(k + 1)} v_{(rk + r + s)(2n - 2i)} z^{k + 1} }}{{q^{2rn} z^2  - q^{ri} v_{r(2n - 2i)} z + 1}}}\\ 
&- \sum\limits_{i = 0}^{2n} {(-1)^i\binom {2n}iq^{si} \frac{{q^{s(2n - 2i) + ri} v_{(r - s)(2n - 2i)} z - v_{s(2n - 2i)} }}{{q^{2rn} z^2  - q^{ri} v_{r(2n - 2i)} z + 1}}}\,,
\end{split}
\end{equation}
\begin{equation}\label{eq.y5iki1h}
\begin{split}
2(p^2 -4q)^{n-1}&\sum\limits_{j = 0}^k {u_{rj + s}^{2n - 1} z^j }  \\
=& \sum\limits_{i = 0}^{2n - 1} {(-1)^i\binom {2n - 1}iq^{si} \frac{{q^{r(2n - 1 + ki)} u_{(rk + s)(2n - 1 - 2i)} z^{k + 2}  - q^{ri(k + 1)} u_{(rk + r + s)(2n - 1 - 2i)} z^{k + 1} }}{{q^{{(2n - 1)r}} z^2  - q^{ri} v_{r(2n - 1 - 2i)} z + 1}}}\\ 
&+ \sum\limits_{i = 0}^{2n - 1} {(-1)^i\binom {2n - 1}iq^{si} \frac{{q^{s(2n - 1 - 2i) + ri} u_{(r - s)(2n - 1 - 2i)} z + u_{s(2n - 1 - 2i)} }}{{q^{{(2n - 1)r}} z^2  - q^{ri} v_{r(2n - 1 - 2i)} z + 1}}}\,,
\end{split}
\end{equation}
\begin{equation}\label{eq.hbol99m}
\begin{split}
2\sum\limits_{j = 0}^k {v_{rj + s}^n z^j }  =& \sum\limits_{i = 0}^n {\binom niq^{si} \frac{{q^{r(n + ki)} v_{(rk + s)(n - 2i)} z^{k + 2}  - q^{ri(k + 1)} v_{(rk + r + s)(n - 2i)} z^{k + 1} }}{{q^{rn} z^2  - q^{ri} v_{r(n - 2i)} z + 1}}}\\ 
&- \sum\limits_{i = 0}^n {\binom niq^{si} \frac{{q^{s(n - 2i) + ri} v_{(r - s)(n - 2i)} z - v_{s(n - 2i)} }}{{q^{rn} z^2  - q^{ri} v_{r(n - 2i)} z + 1}}}\,.
\end{split}
\end{equation}
\end{theorem}
\begin{proof}
Set $(x,y)=(\alpha,\beta)$ and $(f,g)=(1,-1)$ in Lemma \ref{lemma.l8djsvx} to obtain identities \eqref{eq.fxcbpcf} and \eqref{eq.y5iki1h}; and $(f,g)=(1,1)$ to get identity \eqref{eq.hbol99m}.
\end{proof}
In particular, we have
\begin{equation}\label{eq.ooa3gtf}
\begin{split}
2(p^2 -4q)^n\sum\limits_{j = 0}^k {u_j^{2n} z^j }  &= \sum\limits_{i = 0}^{2n} {(-1)^i\binom {2n}i \frac{{q^{2n + ki} v_{k(2n - 2i)} z^{k + 2}  - q^{i(k + 1)} v_{(k + 1)(2n - 2i)} z^{k + 1} }}{{q^{2n} z^2  - q^i v_{2n - 2i} z + 1}}}\\ 
&- \sum\limits_{i = 0}^{2n} {(-1)^i\binom {2n}i \frac{{q^{i} v_{2n - 2i} z - 2 }}{{q^{2n} z^2  - q^i v_{2n - 2i} z + 1}}}\,,
\end{split}
\end{equation}

\begin{equation}\label{eq.fattk27}
\begin{split}
2(p^2 -4q)^{n-1}&\sum\limits_{j = 0}^k {u_j^{2n - 1} z^j }  \\
=& \sum\limits_{i = 0}^{2n - 1} {(-1)^i\binom {2n - 1}i \frac{{q^{2n - 1 + ki} u_{k(2n - 1 - 2i)} z^{k + 2}  - q^{i(k + 1)} u_{(k + 1)(2n - 1 - 2i)} z^{k + 1} }}{{q^{{2n - 1}} z^2  - q^i v_{2n - 1 - 2i} z + 1}}}\\ 
&+ \sum\limits_{i = 0}^{2n - 1} {(-1)^i\binom {2n - 1}i \frac{{q^{i} u_{2n - 1 - 2i} z}}{{q^{{2n - 1}} z^2  - q^i v_{2n - 1 - 2i} z + 1}}}\,,
\end{split}
\end{equation}

\begin{equation}
\begin{split}
2\sum\limits_{j = 0}^k {v_j^n z^j }  =& \sum\limits_{i = 0}^n {\binom ni \frac{{q^{n + ki} v_{k(n - 2i)} z^{k + 2}  - q^{i(k + 1)} v_{(k + 1)(n - 2i)} z^{k + 1} }}{{q^{n} z^2  - q^i v_{n - 2i} z + 1}}}\\ 
&- \sum\limits_{i = 0}^n {\binom ni \frac{{q^{i} v_{n - 2i} z - 2 }}{{q^{n} z^2  - q^i v_{n - 2i} z + 1}}}\,.
\end{split}
\end{equation}
St\u anic\u a \cite{stanica03} obtained results which may be considered equivalent to \eqref{eq.ooa3gtf} and \eqref{eq.fattk27}. 

\medskip

Dropping terms proportional to $z^k$, in the limit as $k$ approaches infinity,  we obtain the generating fuctions given in Corollary \ref{cor.ki2vdrz}.
\begin{corollary}[\textbf{Generating functions for powers of the terms of Lucas sequences with indices in arithmetic progression}]\label{cor.ki2vdrz}
The following identities hold for integers $r$ and $s$ and nonnegative integer $n$:
\begin{equation}
\begin{split}
2(p^2 -4q)^n\sum\limits_{j = 0}^\infty {u_{rj + s}^{2n} z^j }= \sum\limits_{i = 0}^{2n} {(-1)^i\binom {2n}iq^{si} \frac{{-q^{s(2n - 2i) + ri} v_{(r - s)(2n - 2i)} z + v_{s(2n - 2i)} }}{{q^{2rn} z^2  - q^{ri} v_{r(2n - 2i)} z + 1}}}\,,
\end{split}
\end{equation}
\begin{equation}
\begin{split}
2(p^2 -4q)^{n-1}&\sum\limits_{j = 0}^\infty {u_{rj + s}^{2n - 1} z^j }=\sum\limits_{i = 0}^{2n - 1} {(-1)^i\binom {2n - 1}iq^{si} \frac{{q^{s(2n - 1 - 2i) + ri} u_{(r - s)(2n - 1 - 2i)} z + u_{s(2n - 1 - 2i)} }}{{q^{{(2n - 1)r}} z^2  - q^{ri} v_{r(2n - 1 - 2i)} z + 1}}}\,,
\end{split}
\end{equation}
\begin{equation}
\begin{split}
2\sum\limits_{j = 0}^\infty {v_{rj + s}^n z^j }  =\sum\limits_{i = 0}^n {\binom niq^{si} \frac{{-q^{s(n - 2i) + ri} v_{(r - s)(n - 2i)} z + v_{s(n - 2i)} }}{{q^{rn} z^2  - q^{ri} v_{r(n - 2i)} z + 1}}}\,.
\end{split}
\end{equation}
\end{corollary}
We have, in particular,
\begin{equation}\label{eq.ya5rhgv}
\begin{split}
2(p^2 -4q)^n\sum\limits_{j = 0}^\infty {u_j^{2n} z^j }= \sum\limits_{i = 0}^{2n} {(-1)^i\binom {2n}i \frac{{-q^{i} v_{2n - 2i} z + 2 }}{{q^{2n} z^2  - q^{i} v_{2n - 2i} z + 1}}}\,,
\end{split}
\end{equation}
\begin{equation}
\begin{split}
2(p^2 -4q)^{n-1}&\sum\limits_{j = 0}^\infty {u_j^{2n - 1} z^j }=\sum\limits_{i = 0}^{2n - 1} {(-1)^i\binom {2n - 1}i \frac{{q^{i} u_{2n - 1 - 2i} z}}{{q^{{(2n - 1)}} z^2  - q^{i} v_{(2n - 1 - 2i)} z + 1}}}\,,
\end{split}
\end{equation}
\begin{equation}\label{eq.vofjxdm}
\begin{split}
2\sum\limits_{j = 0}^\infty {v_j^n z^j }  =\sum\limits_{i = 0}^n {\binom ni \frac{{-q^{i} v_{n - 2i} z + 2 }}{{q^{n} z^2  - q^{i} v_{(n - 2i)} z + 1}}}\,.
\end{split}
\end{equation}
Popov \cite{popov77} obtained results equivalent to \eqref{eq.ya5rhgv} --- \eqref{eq.vofjxdm}.
\begin{theorem}[\textbf{Sum of powers of the terms of the Horadam sequence with indices in arithmetic progression}]
The following identity holds for integers $r$, $s$ and $k$, nonnegative integer $n$ and arbitrary $z$:
\[
\begin{split}
&2\sum\limits_{j = 0}^k {w_{rj+s}^n z^j } \\
&\quad=z^{k + 2}\sum\limits_{i = 0}^n {(AB)^iq^{si}\binom ni\frac{{A^{n - 2i} \alpha^{s(n - 2i) + rn(k + 1) - rik} \beta^{r(ik + n)}  + B^{n - 2i} \alpha^{r(ik + n)} \beta^{s(n - 2i) + rn(k + 1) - rik} }}{{q^{rn} z^2  - q^{ri} v_{r(n - 2i)}z + 1}} } \\
&\qquad-z^{k + 1}\sum\limits_{i = 0}^n {(AB)^iq^{si}\binom ni\frac{{A^{n - 2i} \alpha^{s(n - 2i) + (rn - ri)(k + 1)} \beta^{ri(k + 1)}  + B^{n - 2i} \alpha^{ri(k + 1)} \beta^{s(n - 2i) + (rn - ri)(k + 1)} }}{{q^{rn} z^2  - q^{ri} v_{r(n - 2i)}z + 1}} } \\
&\qquad-z\sum\limits_{i = 0}^n {(AB)^iq^{si}\binom ni\frac{{A^{n - 2i} \alpha^{s(n - 2i) + ri} \beta^{r(n - i)}  + B^{n - 2i} \alpha^{r(n - i)} \beta^{s(n - 2i) + ri} }}{{q^{rn} z^2  - q^{ri} v_{r(n - 2i)}z + 1}}} \\
&\quad\qquad+\sum\limits_{i = 0}^n {(AB)^iq^{si}\binom ni\frac{{A^{n - 2i} \alpha^{s(n - 2i)}  + B^{n - 2i} \beta^{s(n - 2i)} }}{{q^{rn} z^2  - q^{ri} v_{r(n - 2i)}z + 1}}}\,.
\end{split}
\]
\end{theorem}
\begin{proof}
Set $(x,y)=(\alpha,\beta)$ and $(f,g)=(A,B)$ in Lemma \ref{lemma.l8djsvx}.
\end{proof}
\begin{corollary}[\textbf{Generating function for powers of the terms of the Horadam sequence with indices in arithmetic progression}]
The following identity holds for integers $r$, $s$ and $k$, nonnegative integer $n$ and arbitrary $z$:
\[
\begin{split}
2\sum\limits_{j = 0}^\infty {w_{rj+s}^n z^j }&=-z\sum\limits_{i = 0}^n {(AB)^iq^{si}\binom ni\frac{{A^{n - 2i} \alpha^{s(n - 2i) + ri} \beta^{r(n - i)}  + B^{n - 2i} \alpha^{r(n - i)} \beta^{s(n - 2i) + ri} }}{{q^{rn} z^2  - q^{ri} v_{r(n - 2i)}z + 1}}} \\
&\qquad\qquad+\sum\limits_{i = 0}^n {(AB)^iq^{si}\binom ni\frac{{A^{n - 2i} \alpha^{s(n - 2i)}  + B^{n - 2i} \beta^{s(n - 2i)} }}{{q^{rn} z^2  - q^{ri} v_{r(n - 2i)}z + 1}}}\,.
\end{split}
\]

\end{corollary}

\hrule

\noindent 2010 {\it Mathematics Subject Classification}:
Primary 11B39; Secondary 11B37.

\noindent \emph{Keywords: }
Horadam sequence, Fibonacci number, Lucas number, Lucas sequence, summation identity, generating function, partial sum.

\hrule




\begin{thebibliography}{99}
\bibitem{horadam65} A.~F.~Horadam, Basic properties of a certain generalized sequence of numbers, \emph{The Fibonacci Quarterly} {\bf 3}:3 (1965), 161--176.

\bibitem{koshy} T.~Koshy, \emph{Fibonacci and Lucas numbers with applications}, Wiley-Interscience, (2001).

\bibitem{popov77} B. S. Popov, Generating functions for powers of certain second-order recurrence sequences, \emph{The Fibonacci Quarterly} {\bf 15}:3 (1977), 221--224.

\bibitem{ribenboim} P. Ribenboim, \emph{My numbers, my friends}, Springer-Verlag, New York, (2000).

\bibitem{stanica03} P. St\u anic\u a, Generating functions, weighted and non-weighted sums for powers of secondorder recurrence sequences, \emph{The Fibonacci Quarterly} {\bf 41}:3 (2003), 321--333.

\bibitem{vajda} S.~Vajda, \emph{Fibonacci and Lucas numbers, and the golden section: theory and applications}, Dover Press, (2008).

\bibitem{mathworld_lucas} E. W. Weisstein, \htmladdnormallink{Lucas Sequence}{http://mathworld.wolfram.com/LucasSequence.html}, \emph{MathWorld--A Wolfram Web Resource}, March 2019.

\bibitem{wiki_lucas} Wikipedia contributors, \htmladdnormallink{Lucas sequence}{https://en.wikipedia.org/wiki/Lucas_sequence}, \emph{Wikipedia, The Free Encyclopedia}, November 2018.


\end{thebibliography}
\end{document}